\documentclass[11pt,a4paper]{article}
\usepackage[margin=1.1in]{geometry} %
\usepackage[utf8]{inputenc}
\usepackage{mathtools} %
\usepackage{amssymb}

\usepackage{microtype}
\usepackage{enumerate}
\usepackage{accents} %
\usepackage{graphicx}
\usepackage[table,dvipsnames]{xcolor}

\usepackage[colorlinks, pagebackref=true]{hyperref}
\hypersetup{
  linkcolor=[rgb]{0.3,0.3,0.6},
  citecolor=[rgb]{0.2, 0.6, 0.2},
  urlcolor=[rgb]{0.6, 0.2, 0.2}
}
\hypersetup{breaklinks=true}

\allowdisplaybreaks[1]

\usepackage{amsthm}
\usepackage{thmtools, thm-restate}
\usepackage[capitalize, nameinlink]{cleveref}

\declaretheorem[name=Theorem, parent=section]{theorem}
\declaretheorem[name=Corollary, sibling=theorem]{corollary}

\declaretheorem[name=Lemma, sibling=theorem]{lemma}

\declaretheorem[name=Definition, sibling=theorem, style=definition]{definition}
\declaretheorem[name=Remark, sibling=theorem, style=definition]{remark}

\declaretheorem[name=Example, sibling=theorem, style=definition]{example}

\declaretheorem[name=Claim, sibling=theorem, style=remark]{claim}

\crefname{theorem}{Theorem}{Theorems}
\crefname{lemma}{Lemma}{Lemmas}
\crefname{proposition}{Proposition}{Propositions}
\crefname{corollary}{Corollary}{Corollaries}
\crefname{conjecture}{Conjecture}{Conjectures}
\crefname{observation}{Observation}{Observations}
\crefname{fact}{Fact}{Facts}
\crefname{definition}{Definition}{Definitions}
\crefname{remark}{Remark}{Remarks}
\crefname{problem}{Problem}{Problems}
\crefname{example}{Example}{Examples}
\crefname{question}{Question}{Questions}
\crefname{claim}{Claim}{Claims}
\crefname{section}{Section}{Sections}
\crefname{subsection}{Section}{Sections}
\crefname{subsubsection}{Section}{Sections}

\numberwithin{equation}{section}

\newcommand{\degenleq}{\unlhd}

\DeclareMathOperator{\subrank}{Q}

\DeclareMathAccent{\wtilde}{\mathord}{largesymbols}{"65}

\newcommand{\CC}{\mathbb{C}}
\newcommand{\NN}{\mathbb{N}}

\DeclareMathOperator{\Q}{Q}

\def\<#1>{\left\langle\ignorespaces#1\unskip\right\rangle}
\DeclareMathOperator{\codim}{codim}

\usepackage{scalerel}

\DeclareMathOperator{\GL}{GL}
\DeclareMathOperator{\SL}{SL}

\DeclareMathOperator{\linspan}{span}
\DeclareMathOperator{\rank}{rank} 
\DeclareMathOperator{\col}{col} 
\DeclareMathOperator{\rk}{rk}
\DeclareMathOperator{\row}{row}

\DeclareMathOperator{\im}{im}
\DeclareMathOperator{\diag}{diag}
\DeclareMathOperator{\tr}{tr}

\newcommand{\Comp}{\mathsf{Comp}}
\newcommand{\Alt}{\mathsf{Skew}}

\DeclareMathOperator{\sign}{sign}

\DeclarePairedDelimiterX\braket[2]{\langle}{\rangle}{#1\,\delimsize\vert\,\mathopen{}#2}

\DeclareMathOperator{\id}{Id}

\usepackage{authblk}

\title{Concise tensors with maximal symmetries}

\author{Annika Holtrup}
\author{Jeroen Zuiddam}
\affil{University of Amsterdam}
\date{}

\begin{document}

\maketitle

\vspace{-2em}
\begin{abstract}
Conner, Gesmundo, Landsberg and Ventura (2019) determined the largest stabilizer  dimension of concise $n\times n \times n$ tensors that are binding, and they determined the corresponding maximizing tensors to be the null algebra tensors. They left as an open problem to extend this to all concise $n \times n \times n$ tensors (i.e.~dropping binding). We solve this problem: We prove that the largest stabilizer dimension of concise $n\times n\times n$ tensors is $n^2 + 1$ and the maximizers are the null algebra tensors (as in the binding case) and the skew symmetric tensor $e_1 \wedge e_2 \wedge e_3$. As part of our approach we obtain upper bounds on the stabilizer dimension of matrix tuples under left-right action (generalized Kronecker quiver representations), which we think are of independent interest.
\end{abstract}

 \tableofcontents

\section{Introduction}

One of the most basic parameters of a tensor in $T \in \CC^n \otimes \CC^n \otimes \CC^n$ is its \emph{stabilizer dimension}: the dimension of the subgroup of elements $g \in \GL_n \times \GL_n \times \GL_n$ for which $g\cdot T = T$.\footnote{We discuss other variations from the literature in \cref{rem:conventions}.} Despite its importance and large interest as a tool in understanding and classifying tensors \cite{MR506377, Burgisser1997Degen-6409, MR1781073, MR1773770, MR2932001, conner2021tensorsmaximalsymmetries, MR4204580, MR4644081, MR2865915}, much is still unknown about it, in particular regarding what values it can take and for which tensors. Results on this were first obtained by Conner,  Gesmundo, Landsberg and Ventura~\cite{conner2021tensorsmaximalsymmetries}, motivated by questions in algebraic complexity theory. They proved among other things a tight upper bound on the stabilizer dimension of so-called balanced tensors and classified the maximizing tensors, and left as an open problem to do the same for all concise tensors. We solve this problem (\cref{th:main}), as we will discuss in a moment. In the process we also determine a new upper bound on the stabilizer dimension of matrix tuples under left-right action (generalised Kronecker quiver representations).

Previous work has mostly focused on directions orthogonal to ours, namely the stabilizer dimension of generic tensors (or generic objects in various settings) and studying specific tensors of interest; we give a brief overview:

\emph{Generic tensors.}
For $n\geq4$, a generic $T \in \CC^n \otimes \CC^n \otimes \CC^n$ has stabilizer dimension $2$ which is the smallest value it can possibly have \cite{MR912058, burgisser2015fundamentalinvariantsorbitclosures, bryan2018locallymaximallyentangledstates}. 
The study of stabilizers of points in general positions is a classical subject \cite{MR267040, MR304554} and generic stabilizers are studied more generally in \cite{garibaldi2019genericallyfreerepresentationsi, derksen2020maximumlikelihoodestimationtensor}. 

\emph{Stabilizers of explicit tensors and their dimension.}
The stabilizer of the matrix multiplication tensor is for example studied by de Groote and B\"{u}rgisser--Ikenmeyer \cite{MR506377, burgisser2015fundamentalinvariantsorbitclosures}.
Bernardi, De Lazzari and Gesmundo obtained bounds on the dimension of stabilizer groups for certain graph tensors in \cite{MR4645236} while computing the dimension of the corresponding tensor network varieties.

\emph{Classifying tensor orbits and degenerations.} Stabilizer dimension (because of its invariance and semi-continuity properties) is a natural tool for separating isomorphism classes of tensors (orbits) and determining possible degenerations between tensors \cite{Burgisser1997Degen-6409, MR1781073, MR1773770, MR2865915,  MR4644081}. 

\emph{Symmetric tensors and algebra tensors.}
As a variation, symmetric tensors (polynomials) and their stabilizers under the symmetric action are a subject of interest, for instance for determinant and permanent polynomials (e.g., in the context of VP vs.~VNP \cite{MR2861717}) and immanants \cite{ye2011stabilizerimmanants}
and iterated matrix multiplication polynomials \cite{MR3513064}. 
They similarly appear in classification of algebras
\cite{MR524979}.

\subsection{Basics, examples, previous results}

Let $n \in \NN$ and 
let $T \in V = \CC^n \otimes \CC^n \otimes \CC^n$ be a tensor. Let the group $G = \GL_n \times \GL_n \times \GL_n$ act on $V$ in the natural way. 
We call
\[
G_T \coloneqq \{g \in G : g\cdot T = T\}
\]
the \emph{stabilizer subgroup} of $T$, and $\dim G_T$ the \emph{stabilizer dimension}.  Given an explicit tensor~$T$, computing the stabilizer dimension can be done efficiently via Lie algebras, as follows.
The Lie algebra of $G_T$ \cite[II.2.5, Folgerung~3]{MR768181} is the linear subspace
\[
\mathfrak{g}_T \coloneqq \{ (X,Y,Z) \in (\CC^{n \times n})^{\times 3} : (X \otimes I \otimes I + I \otimes Y \otimes I + I \otimes I \otimes Z)\cdot T = 0 \}.
\]
We write $S \cong T$ if $S \in G\cdot T$ and $S \degenleq T$ if $S \in \overline{G \cdot T}$.
Some basic properties of the stabilizer dimension are:
\begin{lemma}[{\cite[Abschnitt~4]{Burgisser1997Degen-6409}}] Let $S,T \in \CC^n \otimes \CC^n \otimes \CC^n$.
\begin{enumerate}[\upshape (1)]
\item $\dim \mathfrak{g}_T = \dim G_T$.
\item $\dim G_T = \dim G - \dim G\cdot T$. 
\item If $S \cong T$, then $\dim G_S = \dim G_T$.
\item If $S \degenleq T$, then $\dim G_S \geq \dim G_T$. 
\end{enumerate}
\end{lemma}

These are standard facts about algebraic groups and their orbits, and can also be found in \cite[\S4.3, \S8.3, \S9.1]{MR396773}, \cite[II.2.2, II.2.5]{MR768181} and \cite[Ch.~1, \S2]{MR1064110}. Note that $\dim \mathfrak{g}_T$ is the dimension of a linear space, which indeed is easy to compute (for any concretely given $T$).

We will use the following standard notion.

\begin{definition}
	Let $T \in \CC^n \otimes \CC^n \otimes \CC^n$ be a tensor. Write $T = \sum_{i=1}^n e_i \otimes M_i$ so that $M_1,\dots,M_n$ are the slices of $T$ in direction $3$. (We can similarly define the slices in direction 1 and 2.) We call $T$ \emph{concise} if all of the following three conditions hold:
	\begin{itemize}
		\item[(1)] {$\rk (M_1; \dots ; M_n)=n$}
		\item[(2)] {$\rk \begin{pmatrix}
				M_1\\
				\vdots \\
				M_n
			\end{pmatrix}=n$}
		\item[(3)] {$\dim \linspan (M_1,\dots,M_n)=n$.}
	\end{itemize}
(Equivalently, conciseness means that all three so-called flattening ranks are maximal.)
\end{definition}

\begin{remark}[{\cite[Lemma~1.15]{MR4965400}}]
Every tensor is equivalent to a subtensor that is concise.
\end{remark}

\begin{remark}\label{rem:padding}
Let $m \geq n$, let $T \in \CC^n \otimes \CC^n \otimes \CC^n$ be concise, and let $S \in \CC^m \otimes \CC^m \otimes \CC^m$ be obtained from $T$ by padding with zeros. Then one can show that $\dim G_S = \dim G_T + 3m(m-n)$.
\end{remark}

We now consider several examples of tensors and their stabilizer dimensions.

\begin{example}[Diagonal]\label{ex:scalar}\label{ex:unit}
$T = \mathsf{E}_n \coloneqq \sum_{i=1}^n e_i \otimes e_i \otimes e_i \in \CC^n \otimes \CC^n \otimes \CC^n$. Then $\dim G_T = 2n$. Indeed $G_T$ consists of the triples of diagonal matrices $(\diag(a),\diag(b),\diag(c))$ with $a_ib_ic_i = 1$ for all $i$, extended by the diagonally embedded permutation matrices \cite[Theorem~4.3]{burgisser2015fundamentalinvariantsorbitclosures}.%
 Note that $\mathsf{E}_n$ is a direct sum of $n$ copies of $\mathsf{E}_1$. Generally, $\dim G_{S \oplus T} = \dim G_S + \dim G_T$ if $S$ and~$T$ are concise \cite[Lemma~4.2(i)]{Burgisser1997Degen-6409}.
\end{example}

\begin{example}[Matrix multiplication]\label{ex:matmul}
Let $T = \sum_{i,j,k=1}^m e_{i,j} \otimes e_{j,k} \otimes e_{k,i} \in \CC^{m^2} \otimes \CC^{m^2} \otimes \CC^{m^2}$. Writing $T$ as $\tr(XYZ)$, the group $G_T$ is the image of the homomorphism $\GL_m \times \GL_m \times \GL_m \to \GL_{m^2} \times \GL_{m^2} \times \GL_{m^2}$ sending $(A,B,C)$ to $(X \mapsto AXB^{-1},\ Y \mapsto BYC^{-1},\ Z \mapsto CZA^{-1})$ (sometimes called the ``sandwiching action''), a result of de Groote \cite{MR506377} (see also \cite[Theorem~4.5]{burgisser2015fundamentalinvariantsorbitclosures}). Its kernel is $\{(\lambda I, \lambda I, \lambda I) : \lambda \in \CC^*\}$, so $\dim G_T = 3m^2-1$ \cite[Lemma~4.2(ii)]{Burgisser1997Degen-6409}.%
\end{example}

\begin{example}[Null algebra tensor]\label{ex:null-algebra}
$T = \mathsf{N}_n \coloneqq e_1 \otimes e_1 \otimes e_1 + \sum_{i=2}^n (e_i \otimes e_i \otimes e_1 + e_i \otimes e_1 \otimes e_i) \in \CC^n \otimes \CC^n \otimes \CC^n$. Then $\dim G_T = n^2+1$, which follows from an explicit Lie algebra computation \cite[Proposition~3.2]{conner2021tensorsmaximalsymmetries}.
\end{example}

\begin{example}[Skew-symmetric tensor]\label{ex:skew}
	Let
    $T = \mathsf{D} \coloneqq e_1 \wedge e_2 \wedge e_3 = \sum_{\pi \in S_3} \sign(\pi)\, e_{\pi(1)} \otimes e_{\pi(2)} \otimes e_{\pi(3)} \in \CC^3 \otimes \CC^3 \otimes \CC^3.$
    Then $G_T$ contains $\{(\lambda g, \mu g, \nu g) : g \in \SL_3, \ \lambda \mu \nu = 1\}$, which has dimension $8 + 2 = 10$, and in fact  $\dim G_T = 10$ (e.g., \cite[Table~3]{MR1773770}, \cite[Remark~3.4]{conner2021tensorsmaximalsymmetries}, \cite[\S7.1]{MR4644081}). %
\end{example}

\subsection{Binding tensors}

Let $M_1,\dots,M_n$ be the slices of a tensor $T\in \CC^n \otimes \CC^n \otimes \CC^n$ in direction $i$, and $\mathcal{M}=\linspan(M_1,\dots,M_n)$. Then we define the $i$-th \emph{max-rank} of $T$ as $\Q_i(T)= \max \{ \rk M: M \in \mathcal{M} \}$.
We call a tensor $T\in \CC^n \otimes \CC^n \otimes \CC^n$ \emph{binding} if there are $i,j\in[3]$ distinct such that $\Q_i(T)=\Q_j(T)=n$. 
Binding tensors are %
concise \cite{conner2021tensorsmaximalsymmetries}. The term \emph{binding} is due to Bl\"aser--Lysikov \cite{blaeser-lysikov}. (Landsberg--Micha{\l}ek define a tensor $T$ to be $1_A$-generic if $\Q_1(T)=n$, with analogous definitions for $1_B$- and $1_C$-generic tensors \cite{MR3682743}. In this language, a binding tensor is a tensor that fulfills two of the three genericity conditions.) %

\begin{theorem}[Conner--Gesmundo--Landsberg--Ventura {\cite[Proposition~3.2]{conner2021tensorsmaximalsymmetries}}]\label{th:CGLV}
Let $T \in \CC^n \otimes \CC^n \otimes \CC^n$ be concise and binding. Then $\dim G_T \leq n^2 + 1$. Equality holds exactly when $T \cong \mathsf{N}_n$ or a permutation of it.
\end{theorem}

The idea of the proof of \cref{th:CGLV} is that if $T \in \CC^n \otimes \CC^n \otimes \CC^n$ is concise and binding, then it is isomorphic to 
$S = \mathsf{N}_n + U$ %
for some tensor $U \in \CC^n \otimes \CC^n \otimes \CC^n$ supported on $[n] \times \{2,\ldots, n\} \times \{2, \ldots, n\}$ (this goes back to \cite{blaeser-lysikov});
and in turn, $S$ can be degenerated to the null algebra tensor~$\mathsf{N}_n$ (see also \cite{gesmundo-zuiddam}). Consequently, $\dim G_T \leq \dim G_{\mathsf{N}_n} = n^2+1$.

\subsection{Main result}

Conner, Gesmundo, Landsberg and Ventura state as an open problem to determine the largest stabilizer dimension without the binding assumption~\cite[Problem~3.5]{conner2021tensorsmaximalsymmetries}: ``Determine the largest possible dimension of the symmetry group of a concise
tensor. Furthermore, classify concise tensors with symmetry groups of maximal dimension.'' It was already noted in \cite[Remark~3.4]{conner2021tensorsmaximalsymmetries} that there is a non-binding tensor for $n=3$ that also achieves stabilizer dimension $n^2+1 = 10$, namely \cref{ex:skew}.
We solve \cite[Problem~3.5]{conner2021tensorsmaximalsymmetries}:

\begin{theorem}\label{th:main}
Let $T \in \CC^n \otimes \CC^n \otimes \CC^n$ be concise. Then $\dim G_T \leq n^2 + 1$. Equality holds exactly when
\begin{enumerate}[\upshape (1)]
\item $T \cong \mathsf{N}_n$ or a permutation of it, or
\item $T \cong \mathsf{D} \coloneqq e_1 \wedge e_2 \wedge e_3 \in \CC^3 \otimes \CC^3 \otimes \CC^3$.
\end{enumerate}
\end{theorem}
The proof of \cref{th:main} is necessarily very different from that of \cref{th:CGLV}, since a non-binding $T \in \CC^n \otimes \CC^n \otimes \CC^n$ cannot degenerate to the binding tensor $\mathsf{N}_n$ (since max-rank cannot go up in the limit). The idea of our proof is naturally to consider the Lie algebra 
\[
\mathfrak{g}_T \coloneqq \{ (X,Y,Z) \in (\CC^{n \times n})^{\times 3} : (X \otimes I \otimes I + I \otimes Y \otimes I + I \otimes I \otimes Z)\cdot T = 0 \}
\]
which we phrase in terms of the slices $M_1, \ldots, M_n$ of $T$ in direction 3 as
\[
\mathfrak{g}_T = \{(X,Y,Z) : \forall i \in [n], X M_i + M_i Y^T + \sum_{j=1}^n Z_{ij}M_j = 0\}.
\]
After that, we simplify the problem by ``getting rid of the action in third direction'', by proving the upper bound $\dim \mathfrak{g}_T \leq nd + \dim H$ where 
\[
H = \{(X,Y) : \forall i \in [n], X M_i + M_i Y^T + \sum_{j=1}^n Z_{ij}M_j = 0\}.
\]
This is in fact the Lie algebra for the stabilizer of the matrix tuple $(M_1, \ldots, M_d$ under left-right action.
We discuss the idea of how we bound $\dim H$ in \cref{subsec:matrix-tuples}.

\subsection{Stabilizer of matrix tuples via rank profiles}\label{subsec:matrix-tuples}

We prove a statement about tuples of matrices that may be of independent interest. For a matrix tuple $(M_1, \ldots, M_d) \in (\CC^{n\times n})^d$ let
\[
G_{M_1, \ldots, M_d} \coloneqq \{(P,Q) \in \GL_n \times \GL_n : \forall i \in [d], P M_i Q^T = M_i\}
\]
be the stabilizer of the tuple. Let $q$ be the largest rank of any element of the span of $M_1, \ldots, M_d$. 
We consider the $n \times nk$ matrix $M_1 ; M_2 ; \cdots M_k$ obtained by horizontally concatenating the $M_1, \ldots, M_k$ and define the numbers $c_k \coloneqq \rk(M_1 ; M_2 ; \cdots M_k) - \rank(M_1 ; M_2 ; \cdots M_{k-1})$. And we similarly define $r_k$ by concatenating vertically instead. 
We prove the following:

\begin{theorem}\label{th:matrix-tuple}
Let $M_1, \ldots, M_d \in \CC^{n\times n}$ satisfy $\sum_{i=1}^d \col M_i = \sum_{i=1}^d \row M_i = \CC^n$. Suppose $q < n$, the profiles $r_1, \ldots, r_d$ and $c_1, \ldots, c_d$ are non-increasing, and $r_1 = c_1 = q$. Then
\[
\dim G_{M_1, \ldots, M_d} \leq \sum_{i=1}^d r_i c_i - r_2(q-r_2).
\]
\end{theorem}

The assumption that $r_1, \ldots, r_d$ and $c_1, \ldots, c_d$ are non-increasing and $r_1 = c_1 = q$ is satisfied by a generic choice of basis in any $d$-dimensional space of matrices. (For details see \cref{sec:profiles}.)

We prove \cref{th:matrix-tuple} through the Lie algebra of $G_{M_1, \ldots, M_d}$, which has the same dimension and is the aforementioned space $H$.
(see also \cref{sec:filtration}).

In the case of $d = 2$, a pair $(M_1, M_2) \in \CC^{n \times n}$ is often called a matrix pencil, and here the dimension $G_{M_1, M_2}$ is known exactly by work of Demmel and Edelman \cite{MR1355688} using the Kronecker normal form. %
A $d$-tuple of $n \times n$ matrices is studied as a representation of the generalized Kronecker quiver with $d$ arrows and dimension vector $(n,n)$, and $G_{M_1, \ldots, M_d}$ is its automorphism group \cite[Example~9.8.6]{MR3727119}. 
Quiver theory describes the generic value of the stabilizer dimension \cite{MR1162487}. %

\subsection{Related notions and work}
\label{rem:conventions}
For any tensor $T \in \CC^n \otimes \CC^n \otimes \CC^n$ we defined the stabilizer as the subgroup $G_T \coloneqq\{ g : g\cdot T = T \} \subseteq G \coloneqq \GL_n  \times \GL_n \times \GL_n$. It is worth pointing out two natural variations on this definition that have been studied in the literature, the first essentially equivalent to ours (by a fixed shift in dimension), and the second not equivalent to ours:

\begin{enumerate}[(1)]
\item 
\emph{Quotienting trivial symmetries.}
For every tensor $T \in \CC^n \otimes \CC^n \otimes \CC^n$, the stabilizer~$G_T$ contains the two-dimensional subgroup 
\[
K \coloneqq \{(\lambda \id, \mu \id, \nu \id) \in G : \lambda,\mu,\nu \in \CC, \lambda\mu\nu = 1\}.
\]
Some authors (e.g.,~\cite{conner2021tensorsmaximalsymmetries, MR3729273}) consider the elements of $K$ explicitly as trivial symmetries and thus focus on the quotient group~$G_T / K$ instead of $G_T$. 
The dimensions are simply related by $\dim G_T / K = \dim G_T - 2$, so this notion is equivalent to ours up to a constant shift in dimension.

\item
\emph{Special linear groups.}
Instead of considering the stabilizer $G_T \subseteq \GL_n \times \GL_n \times \GL_n$, some authors consider the subgroup $G'_T \coloneqq \{g : g\cdot T = T\} \subseteq \SL_n \times \SL_n \times \SL_n$, where~$\SL_n$ is the group of matrices of determinant one
(e.g., Nurmiev \cite{MR1773770, MR1781073} in his classification of the $\SL_3 \times \SL_3 \times \SL_3$-orbits in $\CC^3 \otimes \CC^3 \otimes \CC^3$). Unlike for the above $G_T / K$, here the dimension $\dim G'_T$ is not a constant shift of our $\dim G_T$. Indeed, $\dim G'_T = \dim G_T - \delta(T)$, where $\delta(T)$ is the dimension of the image of $G_T$ under $(\det, \det, \det)$. This image is a subgroup of $(\CC^*)^3$, so $\delta(T) \leq 3$, and it always contains the image of $K$, which is the two-dimensional subgroup $\{(a,b,c) \in (\CC^*)^3 : abc = 1\}$, so $\delta(T) \geq 2$. Hence $\delta(T) \in \{2,3\}$ and indeed $\delta(T)$ is not constant. For instance, $\delta(\mathsf{D}) = 2$ while $\delta(\mathsf{N}_n) = 3$. %
\end{enumerate}

\section{Matrix subspaces of bounded rank}\label{sec:matrix-subspaces}

Linear spaces of matrices of bounded rank have been studied since Dieudonn\'e \cite{MR29360} and Flanders \cite{MR136618}. For bounded rank up to $4$ they have been classified by Westwick \cite{MR296081} (bounded rank $1$), Atkinson--Lloyd \cite{MR587090, MR621563} (bounded rank $2$), Atkinson \cite{MR695915} (bounded rank $3$), and Eisenbud--Harris \cite{MR954659} (bounded rank $2$ and $3$, with different arguments that work only over algebraically closed fields), and Huang--Landsberg \cite{MR5044121} (bounded rank $4$).

For $a,b \geq 0$ let $\Comp_{a,b} \subseteq \CC^{n \times n}$ (``compression space'') be the subspace of all $n \times n$ matrices of the block form
\[
\begin{bmatrix} A & B \\ C & 0 \end{bmatrix}
\]
with $A$ of size $a \times b$, that is, the matrices supported on the first $a$ rows and the first $b$ columns. %
Every element of $\Comp_{a,b}$ has rank at most $a+b$. Let $\Alt \subseteq \CC^{n\times n}$ (``skew-symmetric matrices'') be the subspace of  matrices
\[
\begin{bmatrix} S & 0 \\ 0 & 0 \end{bmatrix},\, \textnormal{where }  
S = \begin{bsmallmatrix} \!\phantom{-}0 & \phantom{-}x & \phantom{-}y \\ \!-x & \phantom{-}0 & \phantom{-}z \\ \!-y & -z & \phantom{-}0 \end{bsmallmatrix},\, x,y,z \in \CC.
\]
Every element in $\Alt$ has rank at most two.
Subspaces $V, W \subseteq \CC^{n\times n}$ are called \emph{equivalent} if $W = \{PMQ : M \in V\}$ for some $P,Q \in \GL_n$.

We first use the classical classification of matrix subspaces of bounded rank 1.

\begin{lemma}[Westwick {\cite[Theorem~4.2]{MR296081}}, Atkinson--Lloyd {\cite[Lemma~2]{MR621563}}]\label{lem:bounded-rank-one}
	Let $V \subseteq \CC^{n\times n}$ be a subspace of matrices all of which have rank $\leq1$. Then, up to equivalence,
	\[
	V \subseteq \Comp_{1,0} \quad\text{or}\quad V\subseteq\Comp_{0,1}.
	\]
\end{lemma}
From the matrix subspace classification \cref{lem:bounded-rank-one} one obtains almost directly the following tensor theorem.
\begin{lemma}\label{lem:max-rank-at-least-2}
	Let $T \in \CC^n \otimes \CC^n \otimes \CC^n$ concise, $n\geq 2$. Then $\subrank_1(T), \subrank_2(T), \subrank_3(T) \geq 2$. %
\end{lemma}
\begin{proof}
	Let $V$ be the slice span of $T$ in some direction $i$. If $\subrank_i(T) \leq 1$, then by \cref{lem:bounded-rank-one}, up to equivalence, $V \subseteq \Comp_{1,0}$ or $V \subseteq \Comp_{0,1}$. Then we see that $T$ is not concise.
\end{proof}

Next we use the classical classification of matrix subspaces of bounded rank 2:

\begin{lemma}[Atkinson--Lloyd {\cite[Lemma~7]{MR621563}}, Eisenbud--Harris {\cite[Theorem~1.1]{MR954659}}]\label{lem:bounded-rank-two}
	Let $V \subseteq \CC^{n \times n}$ be a subspace of matrices all of which have rank $\leq 2$. Then, up to equivalence,
	\[
	V \subseteq \Comp_{2,0} \quad\text{or}\quad V \subseteq \Comp_{0,2} \quad\text{or}\quad V \subseteq \Comp_{1,1} \quad\text{or}\quad V = \Alt.
	\]
\end{lemma}

Again, the matrix subspace classification \cref{lem:bounded-rank-two} gives rise to a tensor theorem, as follows (closely related to \cite[Lemma~4.3]{gesmundo-zuiddam}):

\begin{lemma}\label{lem:max-rank-equals-2}
	Let $T \in \CC^n \otimes \CC^n \otimes \CC^n$ concise, $n \geq 3$. Suppose $\subrank_3(T) = 2$. Then either
	\begin{itemize}
		\item $\subrank_1(T) = \subrank_2(T) = n$ (binding in direction 1 and 2)
		\item or $T \cong \mathsf{D}$.
	\end{itemize}
\end{lemma}
\begin{proof}
	Let $V$ be the slice span of $T$ (in direction $3$). We apply \cref{lem:bounded-rank-two}. Conciseness and $n\geq 3$ rules out $V \subseteq \Comp_{2,0}$ and $V \subseteq \Comp_{0,2}$.
    
    In the case $V \subseteq \Comp_{1,1}$, we can see that $\subrank_1(T) = \subrank_2(T) = n$. Indeed, let $M_1, \dots, M_n$ be the slices of $T$ in direction $3$, and let $v_{i,j}$ be the $j$-th column of slice $M_i$. Note that $v_{i,j} \in \linspan (e_1)$ for all $j>1$. Assume $\Q_2(T)<n$, then $\dim \linspan(v_{1,1},\dots,v_{n,1})<n$. By conciseness of $T$, we know that $\dim \linspan(v_{i,j}: i,j \in [n])=n$. This forces $\Q_2(T)=n-1$. After appropriate base change, we have $v_{i,1}=e_i$ for all $i>1$, and there exists a $j>1$ such that $v_{1,j}=e_1$. Adding the $j$-th column simultaneously on all slices to the first column yields an isomorphic tensor $T'$ with $\Q_2(T')=n$ which is a contradiction. Hence $\Q_1(T)=\Q_2(T)=n$. 
    
    In the case $V = \Alt$, one can prove that $T \cong \mathsf{D}$.
\end{proof}

\section{Main proof outline}

To prove \cref{th:main} we will prove the following more precise stabilizer dimension upper bound which incorporates the max-rank of the tensor. 
Let $T \in \CC^n \otimes \CC^n \otimes \CC^n$ be concise. Let $q = \subrank_i(T)$ for some $i \in [3]$.

\begin{theorem}\label{th:stab-dim-max-rank}
Suppose $q < n$. Then $\dim G_T \leq n^2 + 1 - (q-2)(n-q)$.\footnote{In fact, this statement holds without the assumption $q < n$.}
\end{theorem}

From \cref{th:stab-dim-max-rank} we directly get:

\begin{corollary}\label{cor:n-squared}
	If $2 < q < n$, then $\dim G_T \leq n^2$.
\end{corollary}

From \cref{cor:n-squared}, together with the classifications in \cref{sec:matrix-subspaces} and the previously known \cref{th:CGLV} for binding tensors, the main result \cref{th:main} follows quickly:

\begin{proof}[Proof of \cref{th:main}]
To prove: $\dim G_T \leq n^2 + 1$ and the extremizers are $\mathsf{N}_n$ (and permutations) and $\mathsf{D}$.

Suppose $n=1$, then $T \cong e_1 \otimes e_1 \otimes e_1$ and $\dim G_T = 2 = n^2 + 1$ by \cref{ex:scalar}.

Suppose $n \geq 2$. Suppose $T$ is binding. Then by \cref{th:CGLV} we know that $\dim G_T \leq n^2 + 1$ and that the extremizers are $\mathsf{N}_n$ and permutations.

Suppose $T$ is not binding. Then there is a max-rank $q < n$. We do some basic dimension bookkeeping: By \cref{lem:max-rank-at-least-2}, $q \geq 2$, so $n \geq 3$. We consider two cases for $q$:
If $q = 2$, then by \cref{lem:max-rank-equals-2}, either $T$ is binding (which we assumed not) or $T \cong \mathsf{D}$, and $\dim G_T = 10 = n^2 + 1$ by \cref{ex:skew}.
If $q \geq 3$, then $\dim G_T \leq n^2$ by \cref{cor:n-squared}. 
\end{proof}

It remains to prove \cref{th:stab-dim-max-rank}.

\section{Row and column profiles}\label{sec:profiles}

Let $T \in \CC^n \otimes \CC^n \otimes \CC^n$ be concise with slices $M_1, \ldots, M_n$ in direction 3. (The following of course works for a slicing in any direction.) Let
\begin{align*}
U_k &\coloneqq \sum_{i=1}^k \col M_i,\quad c_k \coloneqq \dim U_k - \dim U_{k-1}\\
V_k &\coloneqq \sum_{i=1}^k \row M_i,\quad r_k \coloneqq \dim V_k - \dim V_{k-1}.
\end{align*}
We call $c_1, \ldots, c_n$ the \emph{column profile} of $T$, and $r_1, \ldots, r_n$ the \emph{row profile} of $T$. By conciseness of $T$, we have $c_1 + \cdots + c_n = r_1 + \cdots + r_n = n$. One sees directly that the row and column profile are not invariant under basis transformation of $T$. We can choose a ``good basis'' (and in fact those form a non-empty Zariski-open set). (This is related to ``flushing'' and basis-shifting arguments \cite[Lemma~2.4]{MR4965400}, \cite{MR5059810}.)

\begin{lemma}\label{lemma: existance of good basis}
	There is a good basis for $T$ such that
	\begin{itemize}
		\item $c_1 \geq c_2 \geq \cdots \geq c_n$ and $r_1 \geq r_2 \geq \cdots \geq r_n$ and
		\item $c_1 = r_1 = \subrank_3(T)$
	\end{itemize}
\end{lemma}
\begin{proof}
Let $\mathcal{M}$ denote the slice span of $T$ in direction 3.
We will construct a basis for $\mathcal{M}$, one slice at a time. Suppose that $M_1, \ldots, M_{k-1}$ have been chosen, and write $W_{k-1} \coloneqq \linspan \{M_1, \ldots, M_{k-1}\}$. For $M \in \mathcal{M}$, let 
\[
c(M) \coloneqq \dim(U_{k-1} + \col M) - \dim U_{k-1}, \quad r(M) \coloneqq \dim (V_{k-1} + \row M) - \dim V_{k-1}.
\]
We claim that we can choose $M_k \not\in W_{k-1}$ such that $c(M_k)$ and $r(M_k)$ are simultaneously maximised.

Choose a matrix $M$ maximising $c(M)$. Concatenate the columns of $M$ to a fixed basis of $U_{k-1}$ and choose a largest nonzero minor. This minor is a polynomial $f(M)$, nonzero because it does not vanish at the chosen $M$. For any $M' \in \mathcal{M}\setminus W_{k-1}$, if $f(M')\neq 0$, then $c(M') = c(M)$. A nonzero polynomial $g$ detects the maximal value of $r$ in the same way. Finally, since $W_{k-1}$ is a proper subspace of $\mathcal{M}$, some nonzero linear polynomial $h$ vanishes on it. The product $fgh$ is a nonzero polynomial on the complex vector space $\mathcal{M}$, so it is nonzero at some $M_k$. This matrix lies outside $W_{k-1}$ and simultaneously maximises $c(M_k)$ and $r(M_k)$. Repeating the construction gives an ordered basis $M_1, \ldots, M_k$.

At the first step, $U_0 = V_0 = 0$, so the $c(M)$ and $r(M)$ being maximised are both simply $\rank(M)$. Therefore, $c_1 = r_1 = \subrank_3(T)$. 

It remains to prove that the two profiles are nonincreasing. Suppose, for instance, that $c_i < c_{i+1}$. At stage $i$, the later slice $M_{i+1}$ was an admissible choise, and its new column constribution relative to $U_{i-1}$ is at least its contribution relative to the larger space $U_i$:
\[
\dim(U_{i-1} + \col M_{i+1}) - \dim U_{i-1} \geq \dim (U_i + \col M_{i+1}) - \dim U_i = c_{i+1} > c_i.
\]
This contradicts the choise of $M_i$. Hence $c_1\geq \cdots \geq c_n$. The same argument with rows gives $r_1 \geq \cdots \geq r_n$.
\end{proof}

\section{Profile numerics}

Let $T \in \CC^n \otimes \CC^n \otimes \CC^n$ be concise and $q \coloneqq \subrank_3(T)$.

In order to prove \cref{th:stab-dim-max-rank}, we will prove the following:

\begin{theorem}\label{th:dim-stab-good-basis}
Let $T$ be in a good basis. Suppose $q < n$. Let $d \in [n]$ be the largest such that $c_d > 0$ or $r_d > 0$. Then $\dim G_T \leq nd + \sum_{i=1}^n r_i c_i - (q-1)$.
\end{theorem}

Let us first prove \cref{th:stab-dim-max-rank} from \cref{th:dim-stab-good-basis}, using the following auxiliary lemma. %

\begin{lemma}\label{lem:aux}
$d \leq n-q + 1$ and $\sum_{i=1}^n r_i c_i \leq n + q(n-d)$.
\end{lemma}
\begin{proof}
	Without loss of generality, $d = \ell(r) \geq \ell(c)$. Then $r_1, \ldots, r_d \geq 1$ and $r_i = c_i = 0$ for $i > d$. Then $d-1 \leq r_2 + \cdots + r_d = n-q$ gives the first claim. For the second, we have
	\[
	\sum_{i=1}^n r_i c_i = q^2 + \sum_{i=2}^d c_i + \sum_{i=2}^d (r_i - 1)c_i.
	\]
	We use $\sum_{i=2}^d c_i = n-q$ and $\sum_{i=2}^d (r_i - 1)c_i \leq (n-q-d+1)q$ (since $c_i \leq q$ for all $i$), to get
	\[
	\sum_{i=1}^n r_i c_i \leq q^2 + n-q + (n-q-d+1)q = n+q(n-d).\qedhere
	\]
\end{proof}

\begin{proof}[Proof of \cref{th:stab-dim-max-rank}]
	We use \cref{th:dim-stab-good-basis} and then the two items from \cref{lem:aux} to get
	\begin{align*}
		\dim G_T &\leq nd + \sum_{i=1}^n r_i c_i - (q-1) \\
		&\leq nd + n + q(n-d) - (q-1) \\
		&= n + qn - (q-1) + (n-q)d \\
		&\leq n + qn - (q-1) + (n-q)(n-q+1)\\
		&= n^2 + 1 - (q-2)(n-q).\qedhere
	\end{align*}
\end{proof}

\section{Filtration argument}\label{sec:filtration}
We now prove \cref{th:dim-stab-good-basis} in three steps.

\paragraph{Step 1: Reduce third factor.}
Recall that $\dim G_T = \dim \mathfrak{g}_T$ %
where
\[
\mathfrak{g}_T \coloneq \{(X,Y,Z) : (X \otimes I \otimes I + I \otimes Y \otimes I + I \otimes I \otimes Z)\cdot T = 0\}
\]
which, in terms of the slices $M_1,\ldots, M_n$ of $T$ in direction 3, we can write as
\[
\mathfrak{g}_T = \{(X,Y,Z) : \forall i \in [n], X M_i + M_i Y^T + \sum_{j=1}^n Z_{ij}M_j = 0\}.
\]
Let $\pi : \mathfrak{g}_T \to \CC^{d\times n}$ map $(X,Y,Z)$ to the matrix consisting of the first $d$ rows of $Z$. Let $H \coloneqq \{(X,Y) : \forall i \in [d], XM_i + M_i Y^T = 0\}$. We see that if $(X,Y,Z) \in \ker \pi$, then $(X,Y) \in H$. We can also check that $\ker \pi \to H : (X,Y,Z) \mapsto (X,Y)$ is injective. [Indeed, $(X,Y)=0$ implies $\sum_{j=1}^n Z_{ij} M_j=0$ for all $i\in [n]$ and by conciseness of $T$, this implies $Z=0$.] Thus,
\[
\dim \mathfrak{g}_T = \dim \im \pi + \dim \ker \pi \leq nd + \dim H.
\]
It remains to prove $\dim H \leq \sum_{i=1}^n r_i c_i - (q-1)$.

\paragraph{Step 2: Filtration.}
Recall we defined the spaces
$U_k \coloneqq \sum_{i=1}^k \col M_i, 
V_k \coloneqq \sum_{i=1}^k \row M_i$ and $c_k \coloneqq \dim U_k - \dim U_{k-1},  r_k \coloneqq \dim V_k - \dim V_{k-1}$. The definition of $d$ as largest index such that $c_d > 0$ or $r_d > 0$, gives $U_d = U_n = \CC^n$ and $V_d = V_n = \CC^n$. We define
\[
F_k \coloneqq \{(X,Y) \in H : X U_k = 0\}.
\]
We see that
\[
H = F_0 \supseteq F_1 \supseteq \cdots \supseteq F_d = \dots= F_n= 0.
\]
To see $F_d = 0$: Let $(X,Y) \in F_d$. From $U_d = \CC^n$ we get $X = 0$. Then by definition of $H$, $\forall i \in [d], M_iY^T = 0$. Using $V_d = \CC^n$ we find $Y^T = 0$.

Define the map $\theta_k : F_k \to \CC^{n \times n} , (X,Y) \mapsto X M_{k+1}$. One verifies directly that it has kernel $F_{k+1}$. Thus $\dim \im \theta_k = \dim F_k - \dim F_{k+1}$. So we have a telescoping sum $\dim H = \sum_{k=0}^{n-1} \dim \im \theta_k$.

\paragraph{Step 3: Upper bound $\dim \im \theta_k$.}
In the following two claims, we prove an upper bound on  $\dim \im \theta_k$ for all $k$, and a stronger one for $k=0$.
\begin{claim}\label{claim:1}
	$\forall k, \dim \im \theta_k \leq r_{k+1} c_{k+1}$.
\end{claim}
\begin{proof}
Let
\[
P_k \coloneqq \{ v \in \CC^n : M_{k+1} v \in U_k\},\quad Q_k \coloneqq \{w \in (\CC^n)^* : w M_{k+1} \in V_k\}.
\]
Then $\codim P_k = c_{k+1}$ and $\codim Q_k = r_{k+1}$.
Let $N \in \im \theta_k$. One readily checks that $N P_k = 0$, and $Q_k N = 0$: there is $(X,Y) \in F_k$ with $N=XM_{k+1}$, then $NP_k=0$ follows directly. Since $(X,Y) \in H$, we observe that for $k+1 \leq d$, we have $N=XM_{k+1}=-M_{k+1}Y^T$ and for $k+1 >d$, we have $F_k=0$ and hence $N=0$. Both imply $Q_kN=0$. This implies that $\im \theta_k$, after appropriate basis transformation, is only supported on an $r_{k+1} \times c_{k+1}$ submatrix, and the dimension bound follows.
\end{proof}

\begin{claim}\label{claim:2}
	$\dim \im \theta_0 \leq q^2 - (q-1)$.
\end{claim}
\begin{proof}
After appropriate basis transformation, we have
\[
M_1 = 
\begin{pmatrix}
I_q & 0 \\
0 & 0
\end{pmatrix},\,
M_2 = \begin{pmatrix}
M_2^{1,1} & M_2^{1,2}\\
M_2^{2,1} & 0
\end{pmatrix}
\]
where $M_2^{1,2}$ has rank $r_2$ and $M_2^{2,1}$ has rank $c_2$. Note that from this it follows that $q \geq \rank M_2 \geq r_2 + c_2$. From $q<n$ we get $r_2, c_2 \geq 1$ and so $1 \leq r_2 \leq q-1$, which we will need later. Using $XM_1 + M_1 Y^T = 0$ we find that 
\[
\theta_0(X,Y) = XM_1 = \begin{pmatrix}
	P & 0\\
	0 & 0
\end{pmatrix}
\]
for a $q\times q$ matrix $P$. One can then prove that $P( \im M_2^{12} ) \subseteq \im M_2^{12}$ using $XM_2+M_2Y^T=0$. For a fixed $s$-dimensional subspace $S \subseteq \CC^q$, the matrices $P \in \CC^{q \times q}$ satisfying $P(S) \subseteq S$ form a linear subspace of dimension $q^2 - s(q-s)$. Applying this to our setting with $s = r_2$, gives
\[
\dim \im \theta_0 \leq q^2 - r_2(q-r_2)
\]
Now to see $q^2 - r_2(q-r_2) \leq q^2 - (q-1)$ we use $1 \leq r_2 \leq q-1$ which we established earlier.
\end{proof}

Using \cref{claim:1} and \cref{claim:2} we get 
\[
\dim G_T = \dim \mathfrak{g}_T \leq nd + \dim H = nd + \sum_k \dim \im \theta_k \leq nd + \sum_{i=1}^n r_i c_i - (q-1),
\]
which proves \cref{th:dim-stab-good-basis}.

\paragraph{Acknowledgements.} JZ was supported by NWO Vidi grant VI.Vidi.243.195 and ERC Starting Grant 101220349 SPECTRA. This work used the Dutch national e-infrastructure with the support of the SURF Cooperative using grant no.~EINF-16525. AH was supported by NWO M1 grant OCENW.M.21.316. AH and JZ are grateful for hospitality and support by the Simons Institute for the Theory of Computing, UC Berkeley, as part of the semester programme Complexity and Linear algebra.

\bibliographystyle{alphaurl}
\bibliography{refs}

\end{document}